\documentclass[11pt]{article}
\usepackage{amssymb}
\usepackage{amsfonts}
\usepackage{amsmath}
\usepackage{graphicx}
\usepackage{ulem}

\newtheorem{theorem}{Theorem}

\newtheorem{corollary}[theorem]{Corollary}

\newtheorem{proposition}[theorem]{Proposition}

\newenvironment{proof}[1][Proof]{\textbf{#1.} }{\  \rule{0.5em}{0.5em}}
\input{tcilatex}

\begin{document}

\title{On the relation between the distributions of stopping time and
stopped sum with applications}
\date{}
\author{Boutsikas M.V.$^{1}$, Rakitzis A.C.$^{2}$ and Antzoulakos D.L.$^{3}$ 
\\
$^{1}${\small Dept. of Statistics \& Insurance Science, Univ. of Piraeus,
Greece (mbouts@unipi.gr).}\\
$^{2}${\small Dept. of Statistics \& Actuarial-Financial Mathematics, Univ.
of Aegean, Greece (arakitz@aegean.gr).}\\
$^{3}${\small Dept. of Statistics \& Insurance Science, Univ. of Piraeus,
Greece (dantz@unipi.gr).}}
\maketitle

\begin{abstract}
Let $T\ $be a stopping time associated with a sequence of independent random
variables $Z_{1},Z_{2},...$ . By applying a suitable change in the
probability measure we present relations between the moment or probability
generating functions of the stopping time $T$ and the stopped sum $%
S_{T}=Z_{1}+Z_{2}+...+Z_{T}$. These relations imply that, when the
distribution of $S_{T}$\ is known, then the distribution of $T$\ is also
known and vice versa. Applications are offered in order to illustrate the
applicability of the main results, which also have independent interest. In
the first one we consider a random walk with exponentially distributed up
and down steps and derive the distribution of its first exit time from an
interval $(-a,b).$ In the second application we consider a series of samples
from a manufacturing process and we let $Z_{i},i\geq 1$, denoting the number
of non-conforming products in the $i$-th sample. We derive the joint
distribution of the random vector $(T,S_{T})$, where $T$ is the waiting time
until the sampling level of the inspection changes based on a $k$-run
switching rule. Finally, we demonstrate how the joint distribution of $%
(T,S_{T})$ can be used for the estimation of the probability $p$ of an item
being defective, by employing an EM algorithm.

\medskip


\textbf{AMS 2000 subject classification}:

\begin{tabular}{ll}
Primary: & 60G40 \\ 
Secondary: & 60G50, 62E15%
\end{tabular}
\end{abstract}

\bigskip

\section{Introduction}

In several areas of applied science researchers are interested in studying
the time $T$ to take a given action, based on sequentially observed random
variables (rv's) $Z_{1},~Z_{2},~\ldots $ , as well as in the associated
partial sums $S_{n}=Z_{1}+Z_{2}+\ldots +Z_{n},n=1,2,...$ . The waiting time $%
T$ and the corresponding random sum $S_{T}$ are usually referred to as 
\textit{stopping time} and \textit{stopped sum }respectively. Stopping time
problems arise in many diverse scientific areas such as sequential analysis,
quality control, mathematical finance, operations research, biology,
actuarial science, etc. For a gentle introduction to the theory of stopping
times and stopped sums, the interested reader is referred to Karlin and
Taylor (1975). For a more thorough investigation of the theory of stopped
random walks we refer to Gut (2009).

When studying the distribution of $T$ in a sequence of independent and
identically distributed (iid) trials, the stopped sum $S_{T}$ also provides
useful information about the nature of the statistical experiment. The
pioneering work of Abraham Wald (1945) in the area of sequential analysis
established powerful identities that relate the distributional properties of 
$T$ and $S_{T}$. These identities are usually referred to as \textit{Wald's
(fundamental)\ Identity} and \textit{Wald's} (\textit{first}) \textit{%
equation} and they are, respectively, given by 
\begin{equation}
\mathbb{E}(\left( M_{Z}(w)\right) ^{-T}e^{wS_{T}})=1,  \label{eq26}
\end{equation}%
where $M_{Z}(w)=\mathbb{E}(e^{wZ})$, and 
\begin{equation}
\mathbb{E}\left( S_{T}\right) =\mathbb{E}\left( Z\right) \mathbb{E}\left(
T\right) .  \label{eq31}
\end{equation}

In a recent article Antzoulakos and Boutsikas (2007) established a
particular relation between the distributions of $T$ and $S_{T}.$ More
specifically, they considered the waiting time $T_{r}$ until the $r-$th
occurrence of a pattern $\mathcal{E}$ in a sequence of binary trials $%
Z_{1},Z_{2},...$ and the total number of successes $S_{T_{r}}$ observed
until that time, and established a direct method to obtain the joint
probability generating function (pgf) of $(T_{r},S_{T_{r}})$ from the pgf of 
$T_{r}$ only. In this paper we extend the aforementioned result for any
distribution of the $Z_{i}$'s and any stopping time $T$, determining the
joint distribution of $(T,S_{T})$ from the distribution of $T$ or $S_{T}$%
\textbf{.}

The organization of the paper is as follows: In Section 2 we state the main
identities that connect the distributions of $T$ and $S_{T},$ along with the
required theoretical backup. An important part of our work is comprised of
the applications that are presented in Section 3. These applications, not
only serve as an illustration of the applicability of the results of Section
2, but they also have an interest on their own. In the first one we consider
the first exit time $T$ from an interval $(-a,b)$ ($a>0$ or $a=\infty $) of
a random walk $S_{i},$ $i=1,2,...,$ with exponentially distributed up and
down steps$.$ By identifying the distribution of $S_{T}$ we extract an exact
formula for the pgf of the boundary crossing time $T.$ In the second
application we consider a sequence $Z_{i},$ $i=1,2,...,$ of measurements
taken from samples corresponding to lots of products from a manufacturing
process (e.g. number of defective items in each sample). Denoting by $T$ the
waiting time until the sampling level of the inspection changes using a $k$%
-run switching rule associated with $Z_{i}$'s, we obtain the joint pgf of $T$
and $S_{T}$ ($S_{T}$ denotes the total number of defective items observed
until switching) by exploiting the fact that $T$ follows a geometric
distribution of order $k$. Finally, we demonstrate how the joint
distribution of $T$ and $S_{T}$ can be useful in the estimation of the
probability $p$ of an item being defective, by employing an EM algorithm.

\section{Identities connecting the distributions of stopped sum and stopping
time.}

Let $F_{1},F_{2},...$ be a sequence of distributions on $\mathbf{R}$ such
that $\int_{\mathbf{R}}e^{wz}dF_{i}(z)<\infty ,$ $i=1,2,...,$ for every $w$
in an interval $\mathcal{W}$ containing zero. We can always construct a
sequence of independent rv's $Z_{1},Z_{2},...$ \ on a probability space $%
(\Omega ,\mathcal{F},\mathbb{P})$ such that $Z_{i}\sim F_{i},i=1,2,...$ \ .
Moreover, if $F_{i}(\cdot |w)$ denotes the \textit{exponentially tilted} $%
F_{i},$ i.e. $F_{i}(x|w):=\mathbb{E}(e^{wZ_{i}}I_{[Z_{i}\leq x]})/\mathbb{E}%
(e^{wZ_{i}}),~w\in \mathbf{R},$ we can always change $\mathbb{P}$ to a new
probability measure $\mathbb{\tilde{P}}_{w}$ on $(\Omega ,\mathcal{F})$
under which $Z_{1},Z_{2},...$ \ are still independent but now $Z_{i}\sim
F_{i}(\cdot |w),~i=1,2,...$ . A formal construction of the probability space 
$(\Omega ,\mathcal{F},\mathbb{\tilde{P}}_{w})$ is given in the Appendix.

We shall write $\mathbb{\tilde{E}}_{w}(\cdot )$ for the expected value with
respect to the measure $\mathbb{\tilde{P}}_{w}.$ We shall also use the
notation $\mathbb{P}:=\mathbb{\tilde{P}}_{0},\mathbb{E}:=\mathbb{\tilde{E}}%
_{0}.$ It is easy to see that, in the special case when $Z_{1},Z_{2},...$
possess the same density $f$ with respect to $\mathbb{P}$, their density $%
f_{w}$ with respect to $\mathbb{\tilde{P}}_{w}$ is given by 
\begin{equation*}
f_{w}(z)=\frac{e^{wz}f(z)}{\mathbb{E}(e^{wZ_{i}})}.
\end{equation*}

\textbf{Remark. (}\textit{The derivative }$d\mathbb{\tilde{P}}_{w}/d\mathbb{P%
}$\textit{\ on }$\mathcal{F}_{n}$\textbf{)}. Define $\mathcal{F}_{n}=\sigma
(Z_{1},Z_{2},...Z_{n})\subseteq \mathcal{R}^{\mathbf{N}}$ to be the minimal $%
\sigma $-algebra generated by $Z_{1},Z_{2},...Z_{n}$. The sequence $\mathcal{%
F}_{1},\mathcal{F}_{2},...$ is a nondecreasing sequence of $\sigma $%
-algebras in $\mathcal{R}^{\mathbf{N}}.$ The Radon-Nikodym derivative of $%
\mathbb{\tilde{P}}_{w}$ with respect to $\mathbb{P}$ when both are
restricted to $\mathcal{F}_{n}$ is $X_{n}=e^{w(Z_{1}+Z_{2}+...+Z_{n})}/%
\tprod_{i=1}^{n}\mathbb{E}(e^{wZ_{i}})$ (that is, $\mathbb{\tilde{P}}%
_{w}(A)=\int_{A}X_{n}d\mathbb{P},A\in \mathcal{F}_{n}$) and hence 
\begin{equation}
\mathbb{\tilde{E}}_{w}(Y)=\frac{\mathbb{E}(Ye^{w(Z_{1}+Z_{2}+...+Z_{n})})}{%
\tprod_{i=1}^{n}\mathbb{E}(e^{wZ_{i}})}  \label{eq10}
\end{equation}%
for every $\mathcal{F}_{n}$-measurable random variable $Y$. It is worth
mentioning that, even though $\mathbb{P}$ and $\mathbb{\tilde{P}}_{w}$ are
equivalent on every $\mathcal{F}_{n}$, they are mutually singular on $%
\mathcal{F}_{\infty }=\mathcal{R}^{\mathbf{N}}$ when $w\neq 0$ and $%
Z_{1},Z_{2},...$ are identically distributed (that is, there exist disjoint
sets $A,A^{\prime }$ in $\mathcal{R}^{\mathbf{N}}$ such that $\mathbb{\tilde{%
P}}_{w}(A)=1$ and $\mathbb{P}(A^{\prime })=1$). This can be easily seen
since there exists a set $B\in $ $\mathcal{B}(\mathbf{R}\mathbb{)}$ such
that $\mathbb{\tilde{P}}_{w}(Z_{i}\in B)\neq \mathbb{P}(Z_{i}\in B)$ while
(invoking the strong law of large numbers) $\frac{1}{n}%
\sum_{i=1}^{n}I_{[Z_{i}\in B]}$ converges to $\mathbb{\tilde{P}}%
_{w}(Z_{i}\in B)$ on some $A\in \mathcal{R}^{\mathbf{N}}$ with $\mathbb{%
\tilde{P}}_{w}(A)=1$ and to $\mathbb{P}(Z_{i}\in B)$ on some $A^{\prime }\in 
\mathcal{R}^{\mathbf{N}}$ with $\mathbb{P}(A^{\prime })=1$. Since $\mathbb{%
\tilde{P}}_{w}(Z_{i}\in B)\neq \mathbb{P}(Z_{i}\in B)$ we have that $A\cap
A^{\prime }=\varnothing .$ Thus $\mathbb{P}(X_{n}\rightarrow 0)=1$ (see e.g.
Theorem 35.8 in Billingsley (1986)) even though $\mathbb{E}(X_{n})=1$ for
every $n.$ Therefore, in general, there does not exist a Radon-Nikodym
derivative of $\mathbb{\tilde{P}}_{w}$ with respect to $\mathbb{P}$ on $%
\mathcal{R}^{\mathbf{N}}$ and hence $\mathbb{\tilde{P}}_{w}$ cannot be
constructed on $\mathcal{R}^{\mathbf{N}}$ from $\mathbb{P}$ through a
Radon-Nikodym derivative. This fact does not induce any problem since we
have guaranteed the existence of $\mathbb{\tilde{P}}_{w}$ via the Kolmogorov
Existence Theorem (see Appendix).

\bigskip

Let now $T$ be a stopping time associated with the sequence $Z_{1},Z_{2},...$%
, i.e. the set $[T=n]=\{ \omega \in \Omega :T(\omega )=n\}$ belongs to $%
\mathcal{F}_{n}=\sigma (Z_{1},Z_{2},...,Z_{n})$ for every $n=1,2,...,$ \ and
let $S_{T}:=Z_{1}+Z_{2}+...+Z_{T}.$ The next result relates the
distributions of $T$ and $S_{T}$.

\begin{theorem}
Let $T$ be a stopping time associated with the sequence $Z_{1},Z_{2},...$ ,
and let $Y$ be a random variable such that $Y\cdot I_{[T=n]}$ is $\mathcal{F}%
_{n}$-measurable. Then%
\begin{equation}
\mathbb{E}(Ye^{wS_{T}}I_{[T<\infty ]})=\mathbb{\tilde{E}}_{w}(Y%
\prod_{i=1}^{T}\mathbb{E}(e^{wZ_{i}})I_{[T<\infty ]})  \label{eq11}
\end{equation}%
for all real $w$ such that the above expectations exist.
\end{theorem}

\begin{proof}
If $G_{k}:=Ye^{wS_{T}}\sum_{n=1}^{k}I_{[T=n]}$ then $\left \vert
G_{k}\right
\vert \leq \left \vert Y\right \vert e^{wS_{T}}I_{[T<\infty ]}$
a.s. and $\mathbb{E}(|Y|e^{wS_{T}}I_{[T<\infty ]})<\infty ,$ which, by the
Dominated Convergence Theorem (DCT), implies that $\mathbb{E}%
(\lim_{k}G_{k})=\lim_{k}\mathbb{E}(G_{k})$. Thus,%
\begin{equation*}
\mathbb{E}(Ye^{wS_{T}}I_{[T<\infty ]})=\mathbb{E}(\lim_{k\rightarrow \infty
}G_{k})=\lim_{k\rightarrow \infty }\mathbb{E}(G_{k})=\tsum
\limits_{n=1}^{\infty }\mathbb{E}(YI_{[T=n]}e^{wS_{n}}).
\end{equation*}%
By theorems' assumptions, the r.v. $YI_{[T=n]}$ is $\mathcal{F}_{n}$%
-measurable and hence (see (\ref{eq10}) above) $\mathbb{\tilde{E}}%
_{w}(YI_{[T=n]})=\mathbb{E}(YI_{[T=n]}e^{wS_{n}})/\tprod_{i=1}^{n}\mathbb{E}%
(e^{wZ_{i}})$. Therefore,%
\begin{equation*}
\mathbb{E}(Ye^{wS_{T}}I_{[T<\infty ]})=\tsum \limits_{n=1}^{\infty }\mathbb{%
\tilde{E}}_{w}(YI_{[T=n]}\tprod_{i=1}^{n}\mathbb{E}(e^{wZ_{i}}))=\tsum%
\limits_{n=1}^{\infty }\mathbb{\tilde{E}}_{w}(YI_{[T=n]}\tprod_{i=1}^{T}%
\mathbb{E}(e^{wZ_{i}}))
\end{equation*}%
which, invoking again the DCT, leads to (\ref{eq11}) provided that $\mathbb{%
\tilde{E}}_{w}(|Y|\tprod_{i=1}^{T}\mathbb{E}(e^{wZ_{i}})I_{[T<\infty
]})<\infty .$
\end{proof}

The above result can be considered as a version of Wald's Likelihood Ratio
Identity (WLRI, see e.g. Siegmund (1985), or Lai (2004)).

In the sequel we focus on a special use of Equation (\ref{eq11}). Our aim is
to generalize the following result of Antzoulakos and Boutsikas (2007): If $%
Z_{1},Z_{2}$,.. is a sequence of iid binary rv's (trials) with $\mathbb{P}%
(Z_{i}=1)=1-\mathbb{P}(Z_{i}=0)=p$ and $T$ denotes the waiting time (i.e.
the number of trials)\ until a certain pattern $\mathcal{E}$ occurs in $%
Z_{1},Z_{2}$,.. \ then, the joint pgf of $(T,S_{T})$\ follows from the pgf
of $T$\ through the relation%
\begin{equation}
\mathbb{E}(u^{T}w^{S_{T}})=\mathbb{\tilde{E}}_{w}\left(
(u(1-p+pw))^{T}\right)  \label{eq12}
\end{equation}%
for all $w,u\ $in a neighborhood of 0, where the expectation $\mathbb{\tilde{%
E}}_{w}$ is considered under $\mathbb{\tilde{P}}_{w}$ such that $\mathbb{%
\tilde{P}}_{w}(Z_{i}=1)=1-\mathbb{\tilde{P}}_{w}(Z_{i}=0)=\tfrac{pw}{pw+1-p}$%
. The above identity, reveals that, when the distribution of $T$ is known
then the joint distribution of $(T,S_{T})$ is also known. In other words,
the distribution of $T$ uniquely determines the joint distribution of $%
(T,S_{T})$ and consequently the distribution of $S_{T}.$

A generalization of (\ref{eq12}) could refer to any distribution for the $%
Z_{i}$'s and any stopping time $T.$ In addition, an inverse form of (\ref%
{eq12}) could also be very useful implying that the distribution of $S_{T}$
uniquely determines the joint distribution of $(T,S_{T}).$ As it is shown in
the next two corollaries, generalizations of this form can be easily derived
from Equation (\ref{eq11}).

\begin{corollary}
\label{cor1}If $\mathbb{P}(T<\infty )=\mathbb{\tilde{P}}_{w}(T<\infty )=1$
then 
\begin{equation}
\mathbb{E}(u^{T}e^{wS_{T}})=\mathbb{\tilde{E}}_{w}((u\mathbb{E}%
(e^{wZ}))^{T}),  \label{eq13}
\end{equation}%
for all real $u,w$ such that the above expectations exist. In particular, $%
\mathbb{E}(e^{wS_{T}})=\mathbb{\tilde{E}}_{w}(\mathbb{E}(e^{wZ})^{T}).$
\end{corollary}

\begin{proof}
It follows from (\ref{eq11}) by letting $Z_{1},Z_{2},...$ be a sequence of
iid rv's and by setting $Y=u^{T}$ (note that $u^{T}I_{[T=n]}=u^{n}I_{[T=n]}$
is $\mathcal{F}_{n}$-measurable).
\end{proof}

\begin{corollary}
\label{cor2}If there exists a real function $w_{u}$ such that $\mathbb{E}%
(e^{w_{u}Z})=u^{-1}$ and $\mathbb{\tilde{P}}_{w_{u}}$ is a probability
measure with $\mathbb{P}(T<\infty )=\mathbb{\tilde{P}}_{w_{u}}(T<\infty )=1$%
, then 
\begin{equation}
\mathbb{E}(u^{T}e^{xS_{T}})=\mathbb{\tilde{E}}_{w_{u}}(e^{(x-w_{u})S_{T}}),
\label{eq14}
\end{equation}%
for all real $u,x$ such that the above expectations exist. In particular, $%
\mathbb{E}(u^{T})=\mathbb{\tilde{E}}_{w_{u}}(e^{-w_{u}S_{T}}).$
\end{corollary}

\begin{proof}
By setting $Y=u^{T}e^{(x-w_{u})S_{T}}$ we have that the rv $%
YI_{[T=n]}=u^{n}e^{(x-w_{u})(Z_{1}+...+Z_{n})}I_{[T=n]}$ is $\mathcal{F}_{n}$%
-measurable. Therefore by employing (\ref{eq11}) with respect to the
measures $\mathbb{P}$ and $\mathbb{\tilde{P}}_{w_{u}}$ we get 
\begin{equation*}
\mathbb{E}(u^{T}e^{(x-w_{u})S_{T}}e^{w_{u}S_{T}})=\mathbb{\tilde{E}}%
_{w_{u}}(u^{T}e^{(x-w_{u})S_{T}}\mathbb{E}(e^{w_{u}Z})^{T})
\end{equation*}%
which readily leads to (\ref{eq14}) since $(u\mathbb{E}(e^{w_{u}Z}))^{T}=1.$
\end{proof}

It is worth mentioning that, more generally, we can similarly get from (\ref%
{eq11}) that,%
\begin{equation}
\mathbb{E}(Yu^{T}e^{wS_{T}}I_{[T<\infty ]})=\mathbb{\tilde{E}}_{w}(Y(u%
\mathbb{E}(e^{wZ}))^{T}I_{[T<\infty ]})  \label{eq20a}
\end{equation}%
and 
\begin{equation}
\mathbb{E}(Yu^{T}e^{xS_{T}}I_{[T<\infty ]})=\mathbb{\tilde{E}}%
_{w_{u}}(Ye^{(x-w_{u})S_{T}}I_{[T<\infty ]}),  \label{eq20}
\end{equation}%
where $Y$ is a rv such that $YI_{[T=n]}$ is $\mathcal{F}_{n}$-measurable.
The above corollaries imply that, under appropriate conditions, the
distribution of $S_{T}$ uniquely determines the distribution of the stopping
time $T$ and vice versa. Two applications illustrating this fact are
presented in the following section.

\bigskip

\section{Applications}

\subsection{The distribution of the first exit time of a random walk}

Let $Z_{1},Z_{2},...$ be a sequence of (non-degenerate) iid rv's
representing the consecutive jumps of a random walk $S_{n},n=1,2,..$ , that
is, $S_{n}=Z_{1}+Z_{2}+...+Z_{n}.$ Define also the following stopping time%
\begin{equation*}
T=\inf \{n:S_{n}\geq b~\text{or }S_{n}\leq -a\}
\end{equation*}%
for some $a,b>0$. Obviously, $T$ expresses the steps of the random walk\
until it exits the set $(-a,b).$ It can be easily verified that $\mathbb{E}%
(T)<\infty $ (e.g. see Karlin and Taylor (1975), p.264) and thus $T$ is
finite a.s.

Probabilities regarding the first passage, or boundary crossing times arise
in a variety of contexts in applied probability and statistics, such as
sequential analysis, ruin theory, queueing theory, stochastic finance etc.
Usually, it is of interest to evaluate the probability $\mathbb{P}(S_{T}\geq
b)=1-\mathbb{P}(S_{T}\leq -a)$, the distribution of $T$ and $\mathbb{E}%
(T),V(T)$.

In order to illustrate the applicability of identities (\ref{eq13}) and (\ref%
{eq14}) we consider first the case when the jumps $Z_{i}$ are exponentially
distributed (negative or positive with probabilities $p$ and $1-p$
respectively) and deduce explicit formulae for the pgf $\mathbb{E}(u^{T}),$
the joint gf $\mathbb{E}(u^{T}e^{xS_{T}})$, the conditional pgf $\mathbb{E}%
\left( u^{T}|S_{T}\geq b\right) $ and the expected values $\mathbb{E}(T)$
and $\mathbb{E}(T|S_{T}\geq b).$ We also consider the case $a=\infty $
corresponding to a random walk with only an upper barrier, which requires a
different treatment (in this case $T$ is not always a.s. finite).

\subsubsection{Random walk with exponentially distributed up and down steps}

\textbf{(a)}\ Denote by $\mathcal{E}(\theta )$ the exponential distribution
with parameter $\theta >0$. For $i=1,2,...$ , let 
\begin{equation*}
Z_{i}=\left \{ 
\begin{array}{c}
X_{i}~\text{with probability }p~~~~ \\ 
-Y_{i}~\text{with probability }1-p%
\end{array}%
\right.
\end{equation*}%
where $X_{1},X_{2},...$ and $Y_{1},Y_{2},...$ are two sequences of iid rv's
such that $X_{i}\sim \mathcal{E}(\theta _{1}),$ $Y_{i}\sim \mathcal{E}%
(\theta _{2})$. It follows that the pdf $f$ of each $Z_{i}$ is the mixture, $%
f(x)=pf_{1}(x)+(1-p)f_{2}(-x),$ where $f_{i}(x)=\theta _{i}e^{-\theta
_{i}x},~x\geq 0,$ and moment generating function (mgf) given by%
\begin{equation*}
\mathbb{E}(e^{wZ})=p\int_{-\infty }^{\infty
}e^{wx}f_{1}(x)dx+(1-p)\int_{-\infty }^{\infty }e^{wx}f_{2}(-x)dx=\frac{%
p\theta _{1}}{\theta _{1}-w}+\frac{(1-p)\theta _{2}}{\theta _{2}+w}.
\end{equation*}%
Initially, we find the probability $\mathbb{P}(S_{T}\geq b)$ via Wald's
Identity by using a standard technique (see e.g. Karlin and Taylor (1975),
p.265). It can be verified that for $w^{\ast }=(1-p)\theta _{1}-p\theta _{2}$
we have that $\mathbb{E}(e^{w^{\ast }Z})=1,$ and therefore from (\ref{eq13})
(or from (\ref{eq26})) we get $\mathbb{E}(e^{w^{\ast }S_{T}})=\mathbb{\tilde{%
E}}_{w^{\ast }}(\mathbb{E}(e^{w^{\ast }Z})^{T})=\mathbb{\tilde{E}}_{w^{\ast
}}(1^{T})=1.$ Hence, it follows that%
\begin{equation*}
1=\mathbb{E}(e^{w^{\ast }S_{T}})=\mathbb{E}\left( e^{w^{\ast
}S_{T}}|S_{T}\geq b\right) \mathbb{P}(S_{T}\geq b)+\mathbb{E}\left(
e^{w^{\ast }S_{T}}|S_{T}\leq -a\right) (1-\mathbb{P}(S_{T}\geq b))
\end{equation*}%
and by solving with respect to $\mathbb{P}(S_{T}\geq b)$ we get 
\begin{equation}
\mathbb{P}(S_{T}\geq b)=\frac{1-\mathbb{E}\left( e^{w^{\ast
}S_{T}}|S_{T}\leq -a\right) }{\mathbb{E}\left( e^{w^{\ast }S_{T}}|S_{T}\geq
b\right) -\mathbb{E}\left( e^{w^{\ast }S_{T}}|S_{T}\leq -a\right) }.
\label{eq21}
\end{equation}%
Invoking the memoryless property of the exponential distribution we have that%
\begin{eqnarray}
\mathbb{E}\left( e^{wS_{T}}|S_{T}\geq b\right) &=&e^{wb}\mathbb{E}\left(
e^{w(S_{T}-b)}|S_{T}-b\sim \mathcal{E}(\theta _{1})\right) =\frac{\theta _{1}%
}{\theta _{1}-w}e^{wb},  \label{eq15} \\
\mathbb{E}\left( e^{wS_{T}}|S_{T}\leq -a\right) &=&e^{-wa}\mathbb{E}\left(
e^{-w(-a-S_{T})}|-a-S_{T}\sim \mathcal{E}(\theta _{2})\right) =\frac{\theta
_{2}}{w+\theta _{2}}e^{-wa},  \notag
\end{eqnarray}%
and combining the above we deduce that for $w^{\ast }\neq 0,$ 
\begin{equation}
\mathbb{P}\left( S_{T}\geq b\right) =\frac{1-\frac{\theta _{2}e^{-w^{\ast }a}%
}{w^{\ast }+\theta _{2}}}{\frac{\theta _{1}e^{w^{\ast }b}}{\theta
_{1}-w^{\ast }}-\frac{\theta _{2}e^{-w^{\ast }a}}{w^{\ast }+\theta _{2}}}=%
\frac{1-\frac{\theta _{2}e^{-((1-p)\theta _{1}-p\theta _{2})a}}{(1-p)(\theta
_{1}+\theta _{2})}}{\frac{\theta _{1}e^{((1-p)\theta _{1}-p\theta _{2})b}}{%
p(\theta _{1}+\theta _{2})}-\frac{\theta _{2}e^{-((1-p)\theta _{1}-p\theta
_{2})a}}{(1-p)(\theta _{1}+\theta _{2})}}.  \label{eq16}
\end{equation}%
For $w^{\ast }=0$ (i.e. the case where $(1-p)\theta _{1}=p\theta _{2}$) we
can take $w^{\ast }\rightarrow 0$ in the above formula and subsequently
deduce that $\mathbb{P}(S_{T}\geq b)=\frac{\theta _{1}+a\theta _{1}\theta
_{2}}{\theta _{1}+\theta _{2}+(b+a)\theta _{1}\theta _{2}}.$

Next, we derive the mgf of $T$ by employing Corollary \ref{cor2}. A solution 
$w_{u}$ of the equation $\mathbb{E}(e^{wZ})=u^{-1}$ with respect to $w$ is%
\begin{equation}
w_{u}=\tfrac{\theta _{1}-\theta _{2}+u((1-p)\theta _{2}-p\theta _{1})+\sqrt{%
(\theta _{1}-\theta _{2}+u((1-p)\theta _{2}-p\theta _{1}))^{2}+4(1-u)\theta
_{1}\theta _{2}}}{2}.  \label{eq17}
\end{equation}%
The function $w_{u}$ is strictly decreasing for $u\in \lbrack 0,1]$ with $%
w_{0}=\theta _{1},w_{1}=\max \{0,(1-p)\theta _{1}-p\theta _{2}\}$ and thus $%
0<w_{u}<\theta _{1}$ for $u\in (0,1).$ Under the measure $\mathbb{\tilde{P}}%
_{w},$ the pdf $f$ of each $Z_{i}$ takes on the form 
\begin{eqnarray*}
f_{w}(x) &=&\frac{e^{wx}f(x)}{\mathbb{E}(e^{wZ})}=\frac{e^{wx}(p\theta
_{1}e^{-\theta _{1}x}I_{[x\geq 0]}+(1-p)\theta _{2}e^{\theta _{2}x}I_{[x<0]})%
}{\frac{p\theta _{1}}{\theta _{1}-w}+\frac{(1-p)\theta _{2}}{\theta _{2}+w}}
\\
&=&\left \{ 
\begin{array}{c}
c_{w}(\theta _{1}-w)e^{-(\theta _{1}-w)x},~~~~x\geq 0~~~~~~~ \\ 
(1-c_{w})(\theta _{2}+w)e^{-(\theta _{2}+w)(-x)},~x<0%
\end{array}%
\right.
\end{eqnarray*}%
where $c_{w}=\frac{p\theta _{1}}{\theta _{1}-w}(\frac{p\theta _{1}}{\theta
_{1}-w}+\frac{(1-p)\theta _{2}}{\theta _{2}+w})^{-1}~(0<c_{w}<1$ for $%
-\theta _{2}<w<\theta _{1}).$ Hence, under $\mathbb{\tilde{P}}_{w_{u}},u\in
(0,1),$ we still have exponentially distributed up and down jumps, but now
the parameters $p,~\theta _{1}$ and $\theta _{2}$ are substituted by $%
c_{w_{u}}=\frac{up\theta _{1}}{\theta _{1}-w_{u}},(\theta _{1}-w_{u}),$ and $%
(\theta _{2}+w_{u})$ respectively. Again, $T$ is finite $\mathbb{\tilde{P}}%
_{w}$-a.s. Using Corollary \ref{cor2} and (\ref{eq15}) it follows that 
\begin{eqnarray}
\mathbb{E}(u^{T}) &=&\mathbb{\tilde{E}}_{w_{u}}\left( e^{-w_{u}S_{T}}\right)
\notag \\
&=&\mathbb{\tilde{E}}_{w_{u}}\left( e^{-w_{u}S_{T}}|S_{T}\geq b\right) 
\mathbb{\tilde{P}}_{w_{u}}(S_{T}\geq b)+\mathbb{\tilde{E}}_{w_{u}}\left(
e^{-w_{u}S_{T}}|S_{T}\leq -a\right) (1-\mathbb{\tilde{P}}_{w_{u}}(S_{T}\geq
b))  \notag \\
&=&\frac{(\theta _{1}-w_{u})e^{-w_{u}b}}{(\theta _{1}-w_{u})+w_{u}}\mathbb{%
\tilde{P}}_{w_{u}}(S_{T}\geq b)+\frac{(\theta _{2}+w_{u})e^{w_{u}a}}{(\theta
_{2}+w_{u})-w_{u}}(1-\mathbb{\tilde{P}}_{w_{u}}(S_{T}\geq b)).  \label{eq18}
\end{eqnarray}%
Also, using (\ref{eq16}) under the probability measure $\mathbb{\tilde{P}}%
_{w_{u}}$, we get 
\begin{equation}
\mathbb{\tilde{P}}_{w_{u}}\left( S_{T}\geq b\right) =\frac{1-\frac{\theta
_{2}+w_{u}}{(1-c_{w_{u}})(\theta _{1}+\theta _{2})}e^{-\beta _{u}a}}{\frac{%
\theta _{1}-w_{u}}{c_{w_{u}}(\theta _{1}+\theta _{2})}e^{\beta _{u}b}-\frac{%
\theta _{2}+w_{u}}{(1-c_{w_{u}})(\theta _{1}+\theta _{2})}e^{-\beta _{u}a}}
\label{eq19}
\end{equation}%
where $\beta _{u}=(1-c_{w_{u}})(\theta _{1}-w_{u})-c_{w_{u}}(\theta
_{2}+w_{u}).$

Combining (\ref{eq18}) and (\ref{eq19}) we deduce the following proposition.

\begin{proposition}
\label{prop1}Let $S_{n},n=1,2,...$ be a random walk with step distribution $%
F(x)=pF_{1}(x)+(1-p)F_{2}(x),$ where $F_{i}\sim \mathcal{E}(\theta
_{i}),i=1,2,~p\in (0,1).$ If $T$ denotes the time until the random walk
exits $(-a,b),a,b>0$ then the probability generating function of $T$ is
given by 
\begin{equation*}
\mathbb{E}(u^{T})=\frac{\left( \frac{(\theta _{1}-w_{u})}{\theta
_{1}e^{w_{u}b}}-\frac{(\theta _{2}+w_{u})e^{w_{u}a}}{\theta _{2}}\right)
\left( 1-\frac{(\theta _{2}+w_{u})^{2}e^{-\beta _{u}a}}{u(1-p)\theta
_{2}(\theta _{1}+\theta _{2})}\right) }{\frac{(\theta
_{1}-w_{u})^{2}e^{\beta _{u}b}}{up\theta _{1}(\theta _{1}+\theta _{2})}-%
\frac{(\theta _{2}+w_{u})^{2}e^{-\beta _{u}a}}{u(1-p)\theta _{2}(\theta
_{1}+\theta _{2})}}+\frac{(\theta _{2}+w_{u})e^{w_{u}a}}{\theta _{2}},~~u\in
(0,1)
\end{equation*}%
where 
\begin{eqnarray}
\beta _{u} &=&-\sqrt{(\theta _{1}-\theta _{2}+u((1-p)\theta _{2}-p\theta
_{1}))^{2}+4(1-u)\theta _{1}\theta _{2}},  \label{eq27} \\
w_{u} &=&\tfrac{1}{2}\left( \theta _{1}-\theta _{2}+u((1-p)\theta
_{2}-p\theta _{1})-\beta _{u}\right) .  \notag
\end{eqnarray}
\end{proposition}

Note that, for the special case $p=\frac{\theta _{1}}{\theta _{1}+\theta _{2}%
},$ the above generating function can also be derived by employing results
established by Khan (2008).

Apart\textbf{\ }from\textbf{\ }its theoretical interest, the above formula
can also be used for the numerical determination of the distribution of $T$
for given values of the parameters $\theta _{1},\theta _{2},p,a$ and $b,$
since%
\begin{equation}
\mathbb{P}(T=m)=\frac{1}{m!}\left. \frac{d^{m}}{du^{m}}(\mathbb{E}%
(u^{T}))\right\vert _{u=0}.  \label{eq25}
\end{equation}%
In practice, this can be easily accomplished by the use of appropriate
mathematical software (e.g. using the function \texttt{SeriesCoefficient} of
Wolfram Mathematica). In Figure 1 the distribution of $T$ has been pictured
for two sets of values of the parameters. The height of the bars represent
the probabilities $\mathbb{P}(T=m),m=0,1,...,50,$ while the small dots show
the corresponding probabilities estimated by Monte Carlo simulation after $%
10^{5}$ iterations.

\begin{figure}[h]
\begin{center}
\includegraphics[scale=0.3]{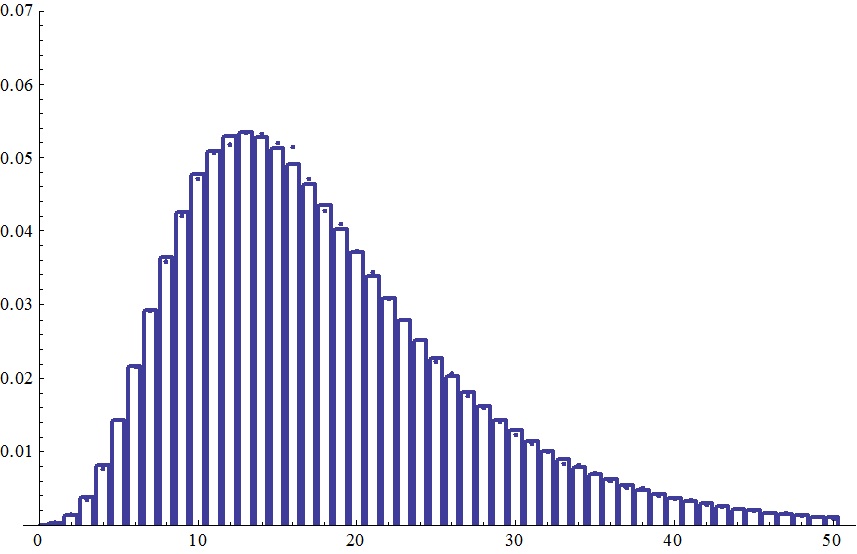} \includegraphics[scale=0.3]{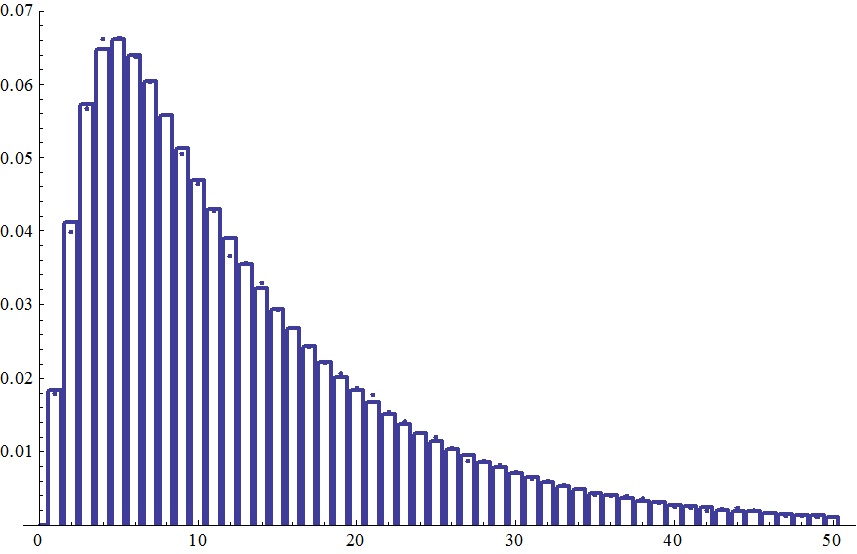}
\end{center}
\caption{The probability mass function of $T$ (${\protect\small p=1/3,}$ $%
{\protect\small \protect\theta }_{1}{\protect\small =2,}$ ${\protect\small 
\protect\theta }_{2}{\protect\small =1,}$ ${\protect\small a=8,}$ $%
{\protect\small b=6}$ and ${\protect\small p=1/2,\protect\theta }_{1}%
{\protect\small =1,\protect\theta }_{2}{\protect\small =1,a=4,b=4}$)}
\end{figure}

An explicit formula for $\mathbb{E}(T)$\ can be easily derived by
differentiating $\mathbb{E}(u^{T})$ given in Proposition \ref{prop1}, with
respect to $u$ and taking $u\rightarrow 1$. The details are left to the
reader.

Employing Corollary \ref{cor2}, we can derive the joint gf of $T$ and $S_{T}$%
, which yields 
\begin{eqnarray*}
\mathbb{E}(u^{T}e^{xS_{T}}) &=&\mathbb{\tilde{E}}_{w_{u}}\left(
e^{xS_{T}}e^{-w_{u}S_{T}}\right) \\
&=&\mathbb{\tilde{E}}_{w_{u}}\left( e^{(x-w_{u})S_{T}}|S_{T}\geq b\right) 
\mathbb{\tilde{P}}_{w_{u}}(S_{T}\geq b)+\mathbb{\tilde{E}}_{w_{u}}\left(
e^{(x-w_{u})S_{T}}|S_{T}\leq -a\right) (1-\mathbb{\tilde{P}}%
_{w_{u}}(S_{T}\geq b)) \\
&=&\frac{(\theta _{1}-w_{u})e^{(x-w_{u})b}}{(\theta _{1}-w_{u})-(x-w_{u})}%
\mathbb{\tilde{P}}_{w_{u}}(S_{T}\geq b)+\frac{(\theta
_{2}+w_{u})e^{(w_{u}-x)a}}{(x-w_{u})+(\theta _{2}+w_{u})}(1-\mathbb{\tilde{P}%
}_{w_{u}}(S_{T}\geq b))
\end{eqnarray*}%
where $\mathbb{\tilde{P}}_{w_{u}}(S_{T}\geq b)$ and $w_{u}$ are given above.

Moreover, for the pgf of the conditional distribution of $T$, given that the
random walk crossed the upper boundary, we observe that (\ref{eq20}) with $%
Y=I_{[S_{T}\geq b]},~x=0,$ leads to%
\begin{eqnarray*}
\mathbb{E}\left( u^{T}|S_{T}\geq b\right) \mathbb{P}\left( S_{T}\geq
b\right) &=&\mathbb{E}(u^{T}I_{[S_{T}\geq b]})=\mathbb{E}_{w_{u}}\left(
e^{-w_{u}S_{T}}I_{[S_{T}\geq b]}\right) \\
&=&\mathbb{\tilde{E}}_{w_{u}}\left( e^{-w_{u}S_{T}}|S_{T}\geq b\right) 
\mathbb{\tilde{P}}_{w_{u}}(S_{T}\geq b).
\end{eqnarray*}%
Therefore we deduce the following result.

\begin{proposition}
\label{prop2}Let $S_{n},n=1,2,...$ be a random walk with step distribution $%
F(x)=pF_{1}(x)+(1-p)F_{2}(x),$ where $F_{i}\sim \mathcal{E}(\theta
_{i}),i=1,2.$ If $T$ denotes the time until the random walk exits $%
(-a,b),a,b>0$ then the conditional pgf of $T$ given that $S_{T}\geq b$, is 
\begin{equation*}
\mathbb{E}\left( u^{T}|S_{T}\geq b\right) =\frac{(\theta
_{1}-w_{u})e^{-w_{u}b}}{\theta _{1}}\cdot \frac{\mathbb{\tilde{P}}%
_{w_{u}}(S_{T}\geq b)}{\mathbb{P}(S_{T}\geq b)},~~u\in (0,1)
\end{equation*}%
where $w_{u},\mathbb{P}(S_{T}\geq b)$ and $\mathbb{\tilde{P}}%
_{w_{u}}(S_{T}\geq b),$ are as in (\ref{eq27}),(\ref{eq16}) and (\ref{eq19})
respectively.
\end{proposition}

Proposition \ref{prop2} along with (\ref{eq25}) can be used for the
calculation of the conditional probabilities $h(m)=\mathbb{P}\left(
T=m|S_{T}\geq b\right) $. In Figure 2, which was constructed similarly to
Figure 1, we have plotted the conditional distribution of $T$ for two sets
of values of the parameters$.$

\begin{figure}[h]
\begin{center}
\includegraphics[scale=0.4]{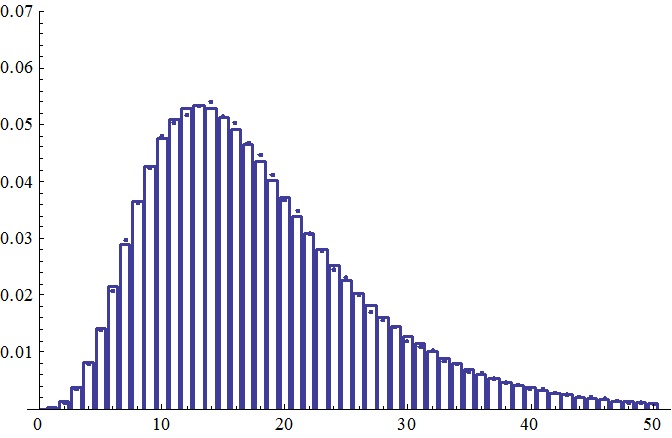} \includegraphics[scale=0.4]{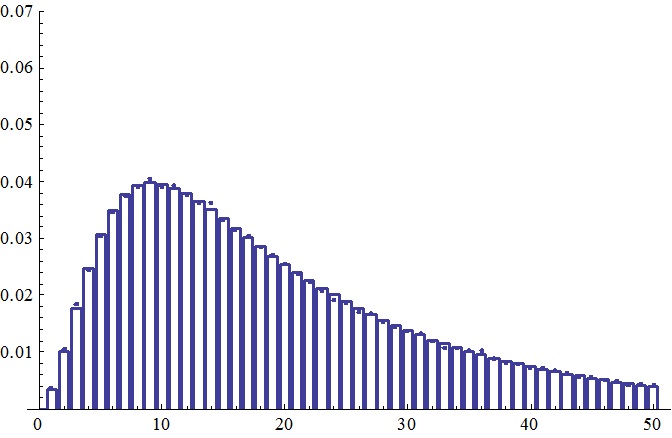}
\end{center}
\caption{The conditional probability mass function $h$ of $T,$ given that $%
S_{T}\geq b.$ (${\protect\small p=2/3,}$ ${\protect\small \protect\theta }%
_{1}{\protect\small =1,}$ ${\protect\small \protect\theta }_{2}%
{\protect\small =2,}$ ${\protect\small a=6,}$ ${\protect\small b=8}$ and $%
{\protect\small p=1/2,\protect\theta }_{1}{\protect\small =1;\protect\theta }%
_{2}{\protect\small =2,a=5,b=5}$)}
\end{figure}

Finally, it is worth mentioning that when the $Z_{i}$'s follow a Laplace
distribution (i.e., $\theta _{1}=\theta _{2}=\theta ,p=1/2)\ $the pgf of $T$
takes on the simple form%
\begin{equation*}
\mathbb{E}(u^{T})=\frac{(e^{a\theta \tilde{u}}+e^{b\theta \tilde{u}})u}{1-%
\tilde{u}+\left( 1+\tilde{u}\right) e^{(a+b)\theta \tilde{u}}}
\end{equation*}%
where $\tilde{u}=\sqrt{1-u}.$ Also, $\mathbb{E}(u^{T}e^{xS_{T}})$ now
simplifies to%
\begin{equation*}
\mathbb{E}(u^{T}e^{xS_{T}})=\frac{(\frac{\theta (1-\tilde{u})e^{(x-\theta 
\tilde{u})b}}{\theta -x}-\frac{\theta (1+\tilde{u})e^{(\theta \tilde{u}-x)a}%
}{x+\theta })(\theta u-(1+\tilde{u})^{2}e^{2\theta \tilde{u}a})}{(1-\tilde{u}%
)^{2}e^{-2\theta \tilde{u}b}-(1+\tilde{u})^{2}e^{2\theta \tilde{u}a}}+\tfrac{%
\theta (1+\tilde{u})}{x+\theta }e^{(\theta \tilde{u}-x)a},
\end{equation*}%
while the conditional pgf of $T$ now reads%
\begin{equation*}
\mathbb{E}\left( u^{T}|S_{T}\geq b\right) =\frac{e^{\theta b\tilde{u}%
}(2+a\theta +b\theta )(1-\tilde{u})(-u+e^{2a\theta \tilde{u}}(2-u+2\tilde{u}%
))}{(1+a\theta )((1-e^{2(a+b)\theta \tilde{u}})(u-2)+2(1+e^{2(a+b)\theta 
\tilde{u}})\tilde{u}}.
\end{equation*}%
Finally, by differentiating $\mathbb{E}(u^{T})$ and $\mathbb{E}\left(
u^{T}|S_{T}\geq b\right) $, with respect to $u$, taking $u\longrightarrow 1$
and after some algebraic manipulations, we may also easily\textbf{\ }derive
explicit formulae for $\mathbb{E}(T)$, $V(T)$\ and $\mathbb{E}\left(
u^{T}|S_{T}\geq b\right) $.

\textbf{(b)} We consider again the random walk $Z_{1},Z_{2},...$ discussed
in (a) with $a=\infty $ (i.e. now there exists only an upper barrier), that
is $T$ denotes the waiting time (steps) until the random walk crosses $b>0$.
Exploiting the results of Section 2, we find the probability $\mathbb{P}%
(T<\infty )$ and the conditional pgf of $T$ given that $T<\infty $. In this
case, $\mathbb{P}(T<\infty )=1$ only when the mean step $\mathbb{E}(Z)=\frac{%
p}{\theta _{1}}-\frac{1-p}{\theta _{2}}$ is positive. We conveniently
observe that the mean step under the probability measure $\mathbb{\tilde{P}}%
_{w_{u}}$ is always positive, that is, 
\begin{equation*}
\mathbb{\tilde{E}}_{w_{u}}(Z)=\frac{c_{w_{u}}}{\theta _{1}-w_{u}}-\frac{%
1-c_{w_{u}}}{\theta _{2}+w_{u}}=\frac{p\theta _{1}u}{(\theta _{1}-w_{u})^{2}}%
-\frac{(1-p)\theta _{2}u}{(\theta _{2}+w_{u})^{2}}>0,
\end{equation*}%
for all $u\in (0,1).$ This can be justified as follows: Note first that $%
w_{u}$ is strictly decreasing for $u\in \lbrack 0,1]$ with $w_{0}=\theta
_{1} $ and $w_{1}=\max \{0,(1-p)\theta _{1}-p\theta _{2}\}.$ It suffices to
show that $g(w_{u})>0,u\in (0,1),$ where $g(x)=(\theta _{2}+x)^{2}p\theta
_{1}-(\theta _{1}-x)^{2}(1-p)\theta _{2}.$ The function $g(x)$ is strictly
increasing in $[0,\theta _{1}]$ ($g^{\prime }(x)>0$ for $x\in \lbrack
0,\theta _{1}])$. We examine the following three cases:

\begin{enumerate}
\item[(i)] If $p\theta _{2}-(1-p)\theta _{1}>0$, then $w_{1}=0$ and hence $%
g(w_{u})>g(w_{1})=g(0)=(p\theta _{2}-(1-p)\theta _{1})\theta _{1}\theta
_{2}>0.$

\item[(ii)] If $p\theta _{2}-(1-p)\theta _{1}<0,$ then $w_{1}=(1-p)\theta
_{1}-p\theta _{2}>0$ and hence $g(w_{u})>g(w_{1})=p(1-p)(\theta _{2}+\theta
_{1})^{2}\left( (1-p)\theta _{1}-p\theta _{2}\right) >0.$

\item[(iii)] If $p\theta _{2}-(1-p)\theta _{1}=0,$ then directly, $\mathbb{%
\tilde{E}}_{w_{u}}(Z)=\tfrac{uw_{u}}{(\theta _{1}-w_{u})(\theta _{2}+w_{u})}%
\left( \tfrac{\theta _{1}}{\theta _{1}-w_{u}}+\tfrac{\theta _{2}}{\theta
_{2}+w_{u}}\right) >0.$\newline
\end{enumerate}

\noindent Therefore, $\mathbb{\tilde{P}}_{w_{u}}(T<\infty )=1,u\in (0,1),$
and from relation (\ref{eq20}) we deduce that%
\begin{eqnarray*}
\mathbb{E}(u^{T}I_{[T<\infty ]}) &=&\mathbb{\tilde{E}}%
_{w_{u}}(e^{-w_{u}S_{T}}I_{[T<\infty ]})=\mathbb{\tilde{E}}%
_{w_{u}}(e^{-w_{u}S_{T}}|T<\infty )\mathbb{\tilde{P}}_{w_{u}}(T<\infty ) \\
&=&\mathbb{\tilde{E}}_{w_{u}}(e^{-w_{u}S_{T}}|T<\infty )=\frac{(\theta
_{1}-w_{u})e^{-w_{u}b}}{\theta _{1}},u\in (0,1).
\end{eqnarray*}%
Letting $u\rightarrow 1$ we get that $\mathbb{P}(T<\infty )=\frac{(\theta
_{1}-w_{1})e^{-w_{1}b}}{\theta _{1}}.$ Since $\mathbb{E}(u^{T}I_{[T<\infty
]})=\mathbb{E}(u^{T}|T<\infty )\mathbb{P}(T<\infty )$ we readily deduce the
following proposition.

\begin{proposition}
\label{prop3}Let $S_{n},n=1,2,...$ be a random walk with step distribution $%
F(x)=pF_{1}(x)+(1-p)F_{2}(x),$ where $F_{i}\sim \mathcal{E}(\theta
_{i}),i=1,2,~p\in (0,1).$ If $T$ denotes the time until the random walk
crosses $b>0,$ then the conditional pgf of $T$ given that $T<\infty $ is 
\begin{equation*}
\mathbb{E}(u^{T}|T<\infty )=\frac{\theta _{1}-w_{u}}{\theta _{1}-w_{1}}%
e^{(w_{1}-w_{u})b},~~u\in (0,1)
\end{equation*}%
where $w_{u}\ $is as in (\ref{eq27}). Moreover,%
\begin{equation*}
\mathbb{P}(T<\infty )=\left \{ 
\begin{array}{c}
\frac{p(\theta _{1}+\theta _{2})}{\theta _{1}}e^{-((1-p)\theta _{1}-p\theta
_{2})b},(1-p)\theta _{1}-p\theta _{2}\geq 0~~ \\ 
1,~~~~~~~~~~~~~~~~~~~~~~~~~~~~~(1-p)\theta _{1}-p\theta _{2}<0.%
\end{array}%
\right.
\end{equation*}
\end{proposition}

By employing Proposition \ref{prop3} we can easily compute the conditional
probabilities $s(m)=\mathbb{P}(T=m|T<\infty )$ through (\ref{eq25}). In
Figure 3 the conditional probabilities $s(m)$ have been plotted for two sets
of values of the parameters.

\begin{figure}[h]
\begin{center}
\includegraphics[scale=0.4]{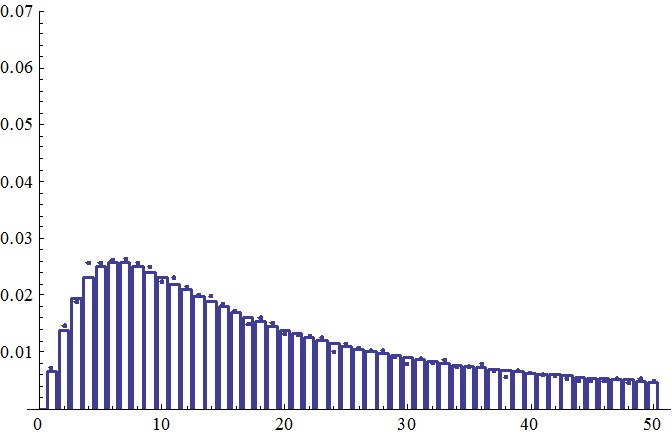} \includegraphics[scale=0.4]{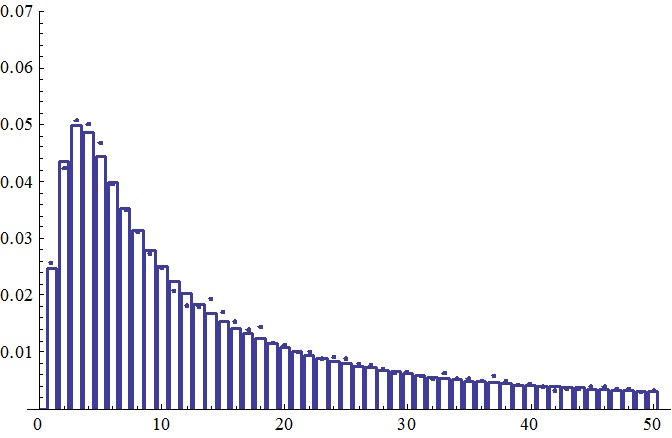}
\end{center}
\caption{The conditional probability mass function $s$ of $T,$ given that $%
T<\infty $ (${\protect\small p=0.4,}$ ${\protect\small \protect\theta }_{1}%
{\protect\small =1.5,}$ ${\protect\small \protect\theta }_{2}{\protect\small %
=2,}$ ${\protect\small b=3}$ and ${\protect\small p=1/2,\protect\theta }_{1}%
{\protect\small =1,\protect\theta }_{2}{\protect\small =1,b=3}$)}
\end{figure}

In the first case we have that $\mathbb{P}(T<\infty )=\frac{14}{15}%
e^{-3/10}\approx 0.69143,$ while in the second case $\mathbb{P}(T<\infty )=1$%
.

\bigskip

\subsection{The distribution of the total number of defective items in a
sampling system based on a $k$-run switching rule.}

In the current paragraph we present an application in acceptance sampling
which is a major component of the field of statistical process control. In
acceptance sampling we frequently deal with sampling systems/plans that have
at least two sampling levels controlled by switching rules that are based on
run and scan statistics. Two examples of such systems are the continuous
sampling plans (see, for example, Schilling and Neubauer (2009)) and the
Military Standard 105E (see, for example, Montgomery (2005)).

In acceptance sampling for attributes we take samples of fixed size
corresponding to consecutive lots of items from a manufacturing process and
we record the number $Z_{i},$ $i=1,2,...$ of non-conforming (defective)
items in the $i$-th sample. Let $c$ be the acceptance number of the
\textquotedblleft normal\textquotedblright \ sampling level, that is a lot
is rejected if the corresponding sample contains more than $c$
non-conforming items. Assume that a switch in a more \textquotedblleft
tightened\textquotedblright \ (\textquotedblleft reduced\textquotedblright )
sampling level is instituted when each one of $k$-consecutive samples have
more than (less than or equal) $c$ non-conforming items. We denote by $T$
the waiting time (i.e. number of lots)\ until the sampling level of the
inspection changes. Our aim is to obtain the joint pgf of $T$ and $S_{T}$ by
exploiting the fact that $T$ follows a known distribution. The study of the
random variable $S_{T}$ is crucial, especially under a rectifying inspection
program.

In the sequel we deal with a sampling system that begins under the normal\
sampling level and a switch is permitted only to the tightened one. More
specifically, suppose that the size of the samples is fixed and equal to $n$
and that the probability of an item being defective is equal to $p\in (0,1).$
Therefore, each $Z_{i},$ $i=1,2,...$ follows a Binomial distribution with
parameters $n,$ $p.$ The number $T$ of inspected lots until the tightened
sampling level is instituted can be expressed as%
\begin{equation*}
T=\inf \{l\geq k:Z_{l-k+1}>c,...,Z_{l}>c\}.
\end{equation*}%
The stopped sum $S_{T}=\sum_{i=1}^{T}Z_{i}$ expresses the total number of
defective items found until switching to the tightened sampling level.

Since $Z_{i}$'s are discrete rv's we can conveniently set $t=e^{w}$ in
Corollary \ref{cor1} to get the following relation for the joint pgf of $%
(T,S_{T}),$%
\begin{equation}
\mathbb{E}(u^{T}t^{S_{T}})=\mathbb{\tilde{E}}_{t}((u\mathbb{E}%
(t^{Z_{1}}))^{T}),  \label{eq23}
\end{equation}%
where $\mathbb{E}(t^{Z_{1}})=(1-p+pt)^{n}$. The distribution of the $Z_{i}$%
's under the probability measure $\mathbb{\tilde{P}}_{t}$ is 
\begin{equation*}
\mathbb{\tilde{P}}_{t}\left( Z_{i}=x\right) =\frac{t^{x}\mathbb{P}(Z_{i}=x)}{%
\mathbb{E}(t^{Z_{1}})}=\binom{n}{x}(\tfrac{pt}{1-p+pt})^{x}(\tfrac{1-p}{%
1-p+pt})^{n-x},~x=0,1,...,n.
\end{equation*}%
Therefore, under $\mathbb{\tilde{P}}_{t},$ $Z_{i}$ follows a binomial
distribution, with parameters $n$ and\thinspace 
\begin{equation}
p_{t}=\frac{pt}{1-p+pt},\text{ }t>0.  \label{eq28}
\end{equation}%
The stopping time $T$\ can be considered as the first time a success run of
length $k$ occurs in a sequence of independent trials with success
probability $q=\mathbb{P}(Z_{i}>c).$ Hence, \thinspace $T<\infty $ and the
distribution of $T$ is known as the\ geometric distribution of order\ $k$
(see, for example, Philippou et al. (1983) or Balakrishnan and Koutras
(2002)) with pgf given by,

\begin{equation}
\mathcal{M}(z,q)=\mathbb{E}(z^{T})=\frac{(qz)^{k}(1-qz)}{%
1-z+(1-q)q^{k}z^{k+1}},~~z\in \lbrack 0,1].  \label{eq24}
\end{equation}%
Under the probability measure $\mathbb{\tilde{P}}_{t}$ we have%
\begin{equation*}
q_{t}=\mathbb{\tilde{P}}_{t}(Z_{i}>c)=1-\sum_{x=0}^{c}\binom{n}{x}%
p_{t}^{x}(1-p_{t})^{n-x},
\end{equation*}%
and thus, $\mathbb{\tilde{E}}_{t}(z^{T}),$ is given by (\ref{eq24}), by
replacing $q$ with $q_{t}.$ Taking into account this observation, equality (%
\ref{eq23}) leads to the following formula for the joint pgf of $(T,S_{T}),$ 
\begin{eqnarray}
\mathbb{E}(u^{T}t^{S_{T}}) &=&\mathbb{\tilde{E}}_{t}((u(1-p+pt)^{n})^{T})=%
\mathcal{M}(u(1-p+pt)^{n},q_{t})  \label{eq29} \\
&=&\frac{(q_{t}u(1-p+pt)^{n})^{k}(1-q_{t}u(1-p+pt)^{n})}{%
1-u(1-p+pt)^{n}+(1-q_{t})q_{t}^{k}(u(1-p+pt)^{n})^{k+1}}  \notag
\end{eqnarray}%
for all $u\in \lbrack 0,1]$ and $t\in (0,1]$ guaranteeing that $%
u(1-p+pt)^{n}\in \lbrack 0,1]$ and $t>0,$ as required by (\ref{eq24}) and (%
\ref{eq28}).

The pgf $\mathbb{E}(t^{S_{T}})$ follows readily from the above by setting $%
u=1.$ The distribution of $S_{T},$ which has support $\{k(c+1),k(c+1)+1,...%
\},$ can be numerically evaluated for specific values of the parameters $%
n,p,c$ and $k$ as described after formula (\ref{eq25}). Using this procedure
we calculate $\mathbb{P}(S_{T}=m)\ $for two sets of the parameters and the
results are shown in Figure 4.

\begin{figure}[h]
\begin{center}
\includegraphics[scale=0.35]{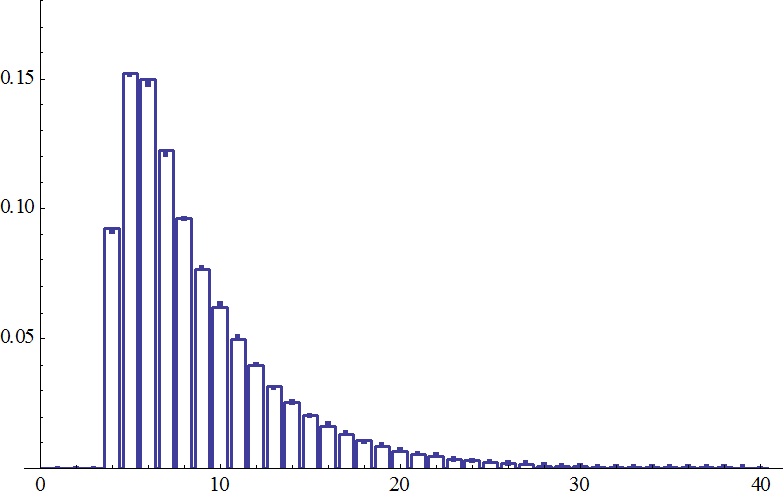} \includegraphics[scale=0.35]{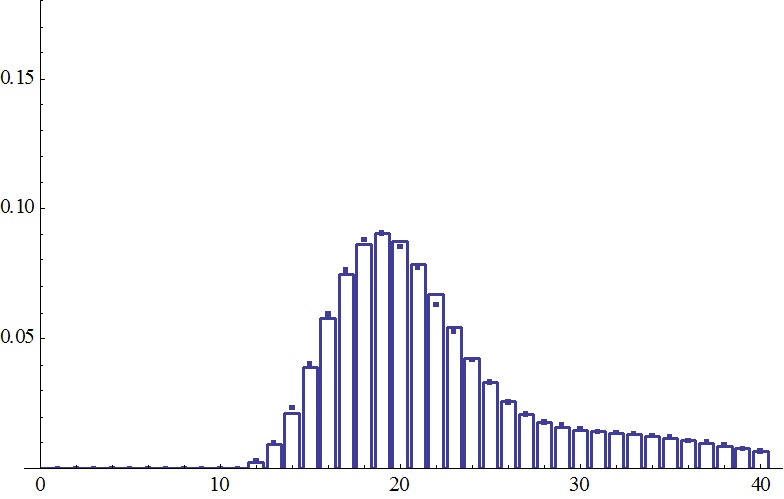}
\end{center}
\caption{The probability mass function of $S_{T}$ (${\protect\small n=20,}$ $%
{\protect\small p=0.1,}$ ${\protect\small c=1,}$ ${\protect\small k=2}$ and $%
{\protect\small n=30,p=0.2,c=3,k=3}$)}
\end{figure}

It should also be mentioned, that since $S_{T}$ is a positive integer-valued
rv, the generating function $\ \mathcal{H}(t)=\sum_{m=0}^{\infty }\mathbb{P}%
(S_{T}>m)t^{m},t\in (-1,1)$\ of the tail probabilities can be easily
determined via the formula 
\begin{equation*}
\mathbb{E}(t^{S_{T}})=1-(1-t)\mathcal{H}(t).
\end{equation*}

The tail probabilities of the distribution of $S_{T}$ can be used in\
practice for the determination of the parameters of the above mentioned
sampling plan. For various combinations of $c$ and $k$ it would be
interesting to know the probability that the total number of defective items
until switching exceeds a certain threshold.\ For example, consider the case
where $n=40,c=1,k=3$ and $p=0.02.$ For $u=1$, Equation (\ref{eq29}) provides
the pgf of $S_{T}$ from which, by differentiation, we get that $\mathbb{E}%
(S_{T})=142.04$ (note that $\mathbb{E}(S_{T})$\ can also be evaluated via
Wald's first equation). Moreover, using $\mathcal{H}(t)$, we can compute the
percentile points of the distribution of $S_{T}$, which provide complete
knowledge about the performance of the sampling plan, in terms of the total
number of defective items found until the switching. Since in that case the
median of the distribution of $S_{T}$\ is 100, we deduce that there is a
probability lower than 50\% that the total number of defective items will
exceed 100 until switching.

It is worth mentioning that the above procedure could easily be expressed in
a more general setting. For example, if the measurements $Z_{i},$ $i=1,2,...$
from the inspected lots follow a general distribution with cdf $F$
(continuous, discrete or mixed) and a switching sampling level occurs at
time $T$ according to some stopping rule (e.g. a $k/m$ scan rule), then
following the methodology described above we can similarly determine the
joint generating function of $(T,S_{T})$ provided that the pgf of $T$ is
known (e.g. is a geometric distribution of order $k/m,$ see Balakrishnan and
Koutras (2002)). In this respect we state without proof the following
proposition.

\begin{proposition}
Let $Z_{i},i=1,2,...$ be a sequence of iid measurements following a
distribution $F$ and let $T$ be the waiting time (i.e. number of $Z_{i}$'s)
until a switching sampling level occurs based on the $k/m$ scan switching
rule: $k$ out of $m$ consecutive $Z_{i}$'s belong to a specific measurable
set $A\subset \mathbf{R}.$ If $\mathcal{M}_{k,m}(z,q)=\mathbb{E}(z^{T}),z\in 
\mathcal{W}$ denotes the pgf of the geometric distribution of order $k/m$
with success probability $q,$ then%
\begin{equation*}
\mathbb{E}(u^{T}e^{wS_{T}})=\mathcal{M}_{k,m}\left( u\mathbb{E}(e^{wZ}),%
\frac{\mathbb{E}(e^{wZ}I(Z\in A))}{\mathbb{E}(e^{wZ})}\right)
\end{equation*}%
for all $u,w$ such that $\mathbb{E}(e^{wZ})<\infty $ and $u\mathbb{E}%
(e^{wZ})\in \mathcal{W}.$
\end{proposition}

The interested reader who wishes to study the general sampling system which
permits a switch from the normal sampling level to the tightened or to the
reduced sampling level may consult Ebneshahrashoob and Sobel (1990) for the
pgf of the associated waiting time rv $T.$

\subsubsection{\textbf{Estimating }$p$\textbf{\ via an EM algorithm.}\ }

In this last subsection we present an interesting application of the formula
of $\mathbb{E}(u^{T}t^{S_{T}})$ obtained above (cf. (\ref{eq29})), regarding
the estimation of the probability $p$ of an item being defective. Assume
that $\nu $ independent inspections are conducted according to the $k$-run
switching rule described above and let $T_{i}$ be the waiting time (i.e.
number of lots)\ until the sampling level of the $i$-th inspection changes, $%
i=1,2,...,\nu $. Denote also by $S_{T_{i}}$ the total number of defective
items found until switching to the tightened sampling level has occurred in
the $i$-th inspection, $i=1,2,...,\nu .$ We are interested in estimating $p$
when only the sample values $\mathbf{\tau }=(\tau _{1},\tau _{2},...,\tau
_{\nu })$ of the $\nu $ aforementioned waiting times are available.

Since the likelihood function $L(p;\mathbf{\tau })=\prod_{i=1}^{\nu }\mathbb{%
P}(T_{i}=\tau _{i}~|~p)\ $does not have a convenient form in order to
directly find the MLE of $p,$ we will show how we can alternatively employ
an EM algorithm, considering $\mathbf{S}_{\tau }=(S_{\tau _{1}},S_{\tau
_{2}},...,S_{\tau _{\nu }})$ as missing values (latent variables). The
likelihood function $L(p;\mathbf{\tau },\mathbf{S}_{\tau })$ now has the
simple form 
\begin{equation*}
L(p;\mathbf{\tau },\mathbf{S}_{\tau })\propto \prod_{i=1}^{\nu }p^{S_{\tau
_{i}}}(1-p)^{n\tau _{i}-S_{\tau _{i}}}=(\frac{p}{1-p})^{\sum_{i=1}^{\nu
}S_{\tau _{i}}}(1-p)^{n\sum_{i=1}^{\nu }\tau _{i}}.
\end{equation*}%
Since $\mathbf{S}_{\tau }$ is not available, we can find the MLE of $p$ by
iteratively applying the following two steps (EM algorithm; cf. Dempster et
al. (1977)):

\textbf{(E-step):} Given $\mathbf{\tau }$ and the estimate of $p$ at the $j$%
-th step, say $p^{(j)}$, compute the conditional expected value of the log
likelihood function, 
\begin{eqnarray*}
Q(p~|~p^{(j)}) &=&\mathbb{E}_{\mathbf{S}_{\tau }|\mathbf{\tau }%
,p^{(j)}}(\log L(p;\mathbf{\tau },\mathbf{S}_{\tau })) \\
&=&\sum_{i=1}^{\nu }\mathbb{E}(S_{T_{i}}|T_{i}=\tau _{i},p^{(j)})\log \frac{p%
}{1-p}+n\sum_{i=1}^{\nu }\tau _{i}\log (1-p).
\end{eqnarray*}%
The expected value $\mathbb{E(}S_{T_{i}}|T_{i}=\tau _{i})$ can be calculated
by 
\begin{equation*}
\mathbb{E}(S_{T_{i}}|T_{i}=\tau _{i})=\frac{1}{\mathbb{P}(T_{i}=\tau _{i})}%
\sum_{m}m\mathbb{P}(S_{T_{i}}=m,T_{i}=\tau _{i}),
\end{equation*}%
which can be derived from Equation (\ref{eq29}). More specifically, the sum $%
\sum_{m}m\mathbb{P}(S_{T}=m,T=r)$ is the coefficient of the $r-$th order
term in the power series expansion of $\mathbb{E}(S_{T}u^{T})$ (where $%
\mathbb{E}(S_{T}u^{T})=\frac{\partial }{\partial t}\mathbb{E}%
(u^{T}t^{S_{T}})|_{t=1}),$ and $\mathbb{P}(T_{i}=\tau _{i})$ can be derived
from the series expansion of $\mathbb{E}(u^{T}).$

\textbf{(M-step): }Find the parameter $p^{(j+1)}$ that maximizes $%
Q(p~|~p^{(j)})$, i.e.,%
\begin{equation*}
p^{(j+1)}=\underset{p}{\arg \max }Q(p~|~p^{(j)})=\frac{\sum_{i=1}^{\nu }%
\mathbb{E}(S_{T_{i}}|T_{i}=\tau _{i},~p^{(j)})}{n\sum_{i=1}^{\nu }\tau _{i}}.
\end{equation*}

The above two steps are repeated until we achieve the desired accuracy in
the estimate $\hat{p}$ of $p$ (e.g. $\hat{p}=p^{(j_{0})}$ where $j_{0}=\min
\{j:|p^{(j)}-p^{(j-1)}|<\varepsilon $). From the above procedure we can also
get an estimate, $\mathbb{E}(S_{T_{i}}|T_{i}=\tau _{i},~\hat{p}),$ of the
unobserved variable $S_{\tau _{i}},i=1,2,...,\nu .$ The observed Fisher
information, which can be exploited for establishing approximate confidence
intervals for $p$, takes on the form 
\begin{eqnarray}
I(\hat{p}) &=&\mathbb{E}_{\mathbf{S}_{T}|\mathbf{\tau },\hat{p}}\left(
-\left. \frac{\partial ^{2}\log L(p;\mathbf{\tau },\mathbf{S}_{\tau })}{%
\partial p^{2}}\right \vert _{p=\hat{p}}\right)  \label{eq34} \\
&=&\frac{1}{(1-\hat{p})^{2}\hat{p}^{2}}\left( (1-2\hat{p})\sum_{i=1}^{\nu }%
\mathbb{E}(S_{T_{i}}|T_{i}=\tau _{i},\hat{p})+n\hat{p}^{2}\sum_{i=1}^{\nu
}\tau _{i}\right) .  \notag
\end{eqnarray}

As an example of the above estimation procedure, suppose that a $k$-run
switching rule is employed for $\nu =20$ inspections with $c=4,k=3,n=50,$
and the resulted waiting times are:%
\begin{equation*}
\mathbf{\tau =(}10,5,17,4,19,3,25,6,16,16,5,4,4,5,6,12,7,12,12,13)
\end{equation*}%
(actually, these are simulated values with $p=0.10$). By employing the EM
algorithm we obtain $\hat{p}=0.0998513$ $(\varepsilon =10^{-8})$ while the
estimates of $S_{\tau _{i}},i=1,2,...,\nu $ are%
\begin{eqnarray*}
&&50,27.4,81.4,22.4,90.3,19.4,117.2,32.4,76.9,76.9, \\
&&27.4,22.4,22.4,27.4,32.4,59,36.4,59,59,63.5
\end{eqnarray*}

The estimated standard error of $\hat{p}$\ is $I(\hat{p})^{-1/2}=0.00299055$
(cf. (\ref{eq34})) and the approximate $1-a=95\%$ confidence interval for $p$
is $(\hat{p}\pm I(\hat{p})^{-1/2}z_{a/2})=(0.0939898,0.105713)$.

\bigskip

\textbf{Acknownledgement}

The work of Athanasios C. Rakitzis is supported by the State Scholarship
Foundation of Greece.

\section{Appendix}

\textbf{The formal construction of} $(\Omega ,\mathcal{F},\mathbb{\tilde{P}}%
_{w}):$ Denote by $\Omega =\mathbf{R}^{\mathbf{N}}$ the collection of all
maps from $\mathbf{N}=\{1,2,...\}$ to $\mathbf{R}$. Each element $\mathbf{x}$
of the product space $\mathbf{R}^{\mathbf{N}}$ can be written as a sequence $%
\mathbf{x}=(x_{1},x_{2},...)$ with each $x_{i}$ belonging to $\mathbf{R}.$
For each $i\in \mathbf{N}$ consider the mapping $Z_{i}:$ $\mathbf{R}^{%
\mathbf{N}}\rightarrow \mathbf{R}$ with $Z_{i}(\mathbf{x})=x_{i}$ (that is, $%
Z_{i}$ is a coordinate function or projection). Let $\mathcal{F=R}^{\mathbf{N%
}}$ be the minimal $\sigma $-algebra such that $Z_{1},Z_{2},...$ are
measurable, i.e. $\mathcal{R}^{\mathbf{N}}:=\sigma (Z_{1},Z_{2},...)=\sigma
(\{ \mathbf{x}\in \mathbf{R}^{\mathbf{N}}:x_{i}\in B\},B\in \mathcal{B}(%
\mathbf{R}\mathbb{)},i\in \mathbf{N)},$ where $\mathcal{B}(\mathbf{R}\mathbb{%
)}$ is the $\sigma $-algebra of the Borel sets of $\mathbf{R}$. Next, denote
by $\mu _{i}$ the probability measure on $\mathcal{B}(\mathbf{R}\mathbb{)}$
that corresponds to $F_{i},i=1,2,...$ . For every $i=1,2,...$ define the
distribution $F_{i}(\cdot |w)$ on $\mathbf{R},$ such that 
\begin{equation*}
F_{i}(x|w):=\frac{\int_{(-\infty ,x]}e^{wz}dF_{i}(z)}{\int_{\mathbf{R}%
}e^{wz}dF_{i}(z)},~~~x\in \mathbf{R},~~w\in \mathcal{W},
\end{equation*}%
which can be considered as the \emph{exponentially tilted} $F_{i}$.
Obviously, $F_{i}(x|0)=F_{i}(x)$. If $\mu _{i}^{w}$ denotes the probability
measure on $\mathcal{B}(\mathbf{R}\mathbb{)}$ corresponding to $F_{i}(\cdot
|w)$ then, equivalently, $\mu _{i}^{w}(B)=\int_{B}e^{wx}\mu _{i}(dx)/\int_{%
\mathbf{R}}e^{wx}\mu _{i}(dx)~$for every $B\in \mathcal{B}(\mathbf{R}\mathbb{%
)}$. Therefore $\mu _{i}^{w}<<\mu _{i}$ and the Radon-Nikodym derivative for 
$\mu _{i}^{w}$ with respect to $\mu _{i}$ reads%
\begin{equation*}
\frac{d\mu _{i}^{w}}{d\mu _{i}}(x)=\frac{e^{wx}}{\int_{\mathbf{R}}e^{wx}\mu
_{i}(dx)},~~~x\in \mathbf{R},~~w\in \mathcal{W}.
\end{equation*}%
Finally, invoking Kolmogorov's Existence Theorem, there exists a probability
measure $\mathbb{\tilde{P}}_{w}$ on $\mathcal{R}^{\mathbf{N}}$ such that the
coordinate variable process $Z_{1},Z_{2},...$ on $(\mathbf{R}^{\mathbf{N}},%
\mathcal{R}^{\mathbf{N}},\mathbb{\tilde{P}}_{w})$ consists of independent
rv's, with distributions $\mu _{1}^{w},\mu _{2}^{w},...$ respectively, and
the construction is completed for all $w\in \mathcal{W}$.

\end{document}